\documentclass[a4paper,10pt]{amsart}

\usepackage[english]{babel}
\usepackage{dsfont} 

\usepackage{amsmath}
\usepackage{amssymb} 
\usepackage{amsthm} 

\newtheorem{thm}{Theorem}
\newtheorem{dfn}[thm]{Definition}
\newtheorem{lem}[thm]{Lemma}

\newtheorem{exa}[thm]{Example}
\newtheorem{prop}[thm]{Proposition}

\newtheorem{cor}[thm]{Corollary}

\newcommand{\C}{\mathds{C}}
\newcommand{\R}{\mathds{R}}
\newcommand{\N}{\mathds{N}}

\newcommand{\cB}{\mathcal{B}}
\newcommand{\cC}{\mathcal{C}}
\newcommand{\cF}{\mathcal{F}}
\newcommand{\cJ}{\mathcal{J}}

\newcommand{\cN}{\mathcal{N}}
\newcommand{\cM}{\mathcal{M}}
\newcommand{\cD}{\mathcal{D}}

\newcommand{\cU}{\mathcal{U}}
\newcommand{\cK}{\mathcal{K}}

\newcommand{\cT}{\mathcal{T}}
\newcommand{\cI}{\mathcal{I}}

\newcommand{\cZ}{\mathcal{Z}}
\newcommand{\cE}{\mathcal{E}}

\newcommand{\ux}{\underline{x}}

\newcommand{\ov}{\overline}
\newcommand{\ii}{\sf{i}}

\newcommand{\Z}{\mathds{Z}}

\newcommand{\cA}{\mathcal{A}}

\newcommand{\cL}{\mathcal{L}}

\newcommand{\Lin}{\rm{Lin}}

\author{Konrad Schm\"udgen}
\address{Universit\"at Leipzig, Mathematisches Institut, Augustusplatz 10/11, D-04109 Leipzig, Germany}
\email{schmuedgen@math.uni-leipzig.de}

\begin{title}
{A general fibre theorem for moment problems and some applications}
\end{title}
\date{}
\begin{document}

\maketitle

\begin{abstract}
The fibre theorem \cite{schm2003} for the moment problem on closed semi-algebraic subsets of $\R^d$ is generalized to finitely generated real unital algebras. As an application two new theorems on the rational multidimensional moment problem are proved. Another application is a characterization of moment functionals on the polynomial algebra $\R[x_1,\dots,x_d]$ in terms of extensions. Finally, the fibre theorem and the extension theorem are used to reprove basic results on the complex moment problem due to Stochel and Szafraniec \cite{stochelsz} and Bisgaard \cite{bisgaard}.

\end{abstract}


\textbf{Key  words:} moment problem, real algebraic geometry

\section {Introduction}

This paper deals with the classical multidimensional moment problem. A useful result on the existence of solutions is the fibre theorem \cite{schm2003} for  closed semi-algebraic subsets of $\R^d$. The crucial assumption for this theorem is the existence of sufficiently many {\it bounded} polynomials on the semi-algebraic set. The proof given  in \cite{schm2003} was based on the decomposition theory of states on $*$-algebras. An elementary proof has been found   by T. Netzer \cite{netzer}, see also    M. Marshall \cite[Chapter 4]{marshall}.

In the present paper we  generalize the fibre theorem from the polynomial algebra $\R[x_1,\dots,x_d]$ to arbitrary finitely generated unital real algebras and we add further statements (Theorem \ref{fibreth1}). This general result and these additional statements allow us to derive a number of  new applications. The first application developed in Section \ref{rationalmp} is about  the rational moment problem on semi-algebraic subsets of $\R^d$.  Pioneering work on this problem was done by J. Cimpric, M. Marshall, and T. Netzer in \cite{cmn}. The main assumption therein is  that the preordering is Archimedean. We essentially use the fibre theorem to go beyond the Archimedean case and  derive  two basic  results on the rational multidimensional moment problem (Theorems \ref{rationalmp2} and \ref{rationalmp1}). The second application given in Section \ref{extensionsection} concerns characterizations of moment functionals in terms of extension to some appropriate larger algebra. We prove an  extension theorem (Theorem \ref{extensionthm}) that provides a necessary and sufficient condition for a linear functional on $\R[x_1,\dots,x_d]$ being a moment functional. In the case $d=2$ this result leads to a  theorem of J. Stochel and F. H. Szafraniec \cite{stochelsz} on the complex moment problem (Theorem  \ref{extcmpstochelsz}). A third application of the general fibre theorem is a  very short proof of a theorem of T.M. Bisgaard \cite{bisgaard} on the two-sided complex moment problem (Theorem \ref{bisgaardthm}).

Let us fix some definitions and notations which are used throughout this paper.
A {\it complex $*$-algebra} $\cB$ is a complex algebra equipped with
an involution, that is,
an antilinear mapping $\cB\ni b\to b^*\in \cB$ satisfying $(bc)^*= c^*b^*$ and $(b^*)^*=b$  for\, $b,c\in \cB$.
Let $\sum \cB^2$ denote the set of finite sums $\sum_j b_j^*b_j$ of hermitian squares $b_j^*b_j$, where $b_j\in \cB$. A linear functional $L$ on $\cB$ is called {\it positive} if $L$ is nonnegative on $\sum \cB^2$, that is, if\, $L(b^*b)\geq 0$\, for all $b\in \cB$.

A  {\it $\ast$-semigroup}\, is a semigroup $S$ with a mapping\, $s^*\to s$\, of\, $S$\, into itself, called  involution, such that $(st)^*=t^*s^*$ and $(s^*)^*=s$ for $s,t\in S$. The {\it semigroup $\ast$-algebra} $\C[ S]$ of $S$ is the complex $*$-algebra is the vector space  of all finite sums $\sum_{s\in S} \alpha_s s$, where $\alpha_s\in \C$, with product and involution
\begin{align*}
\big(\sum\nolimits_s\alpha_s s\big)\big(\sum\nolimits_{t} \beta_t t\big):= \sum\nolimits_{s,t} \alpha_s\beta_t st, \quad \big(\sum\nolimits_{s} \alpha_s s\big)^*:=\sum\nolimits_{s} \ov{\alpha}_s\, s^*.
\end{align*}
The polynomial algebra $\R[x_1,\dots,x_d]$ is abbreviated by $\R_d[\ux]$. By a finite real algebra we mean a real algebra which is finite as a real vector space.

\section{A generalization of the fibre theorem}\label{propertiessmpandmpandfibretheorem}

Throughout this section  $\cA$ is a {\it finitely generated commutative real unital algebra}.  By a character  of $\cA$ we mean an algebra homomorphism $\chi:\cA\to \R$ satisfying $\chi(1)=1$. We equip the set $\hat{\cA}$ of  characters of $\cA$ with the weak topology.

Let us fix a set $\{p_1,\dots,p_d\}$  of  generators of the algebra $\cA$. Then there is a algebra homomorphism $\pi:\R[\ux]\to \cA$ such that $\pi(x_j)=p_j$,  $j=1,\dots,d.$ If\, $\cJ$ denotes the kernel of $\pi$, then  $\cA$ is isomorphic to the quotient algebra\, $\R_d[\ux]/ \cJ$.
Further, each character $\chi$ is completely determined by  the point $x_\chi:=(\chi(p_1),\dots,\chi(p_d))$ of $\R^d$. For simplicity we will identify $\chi$ with $x_\chi$ and write $f(x_\chi):=\chi(f)$ for $f\in \cA$. Under this identification\, $\hat{\cA}$ is a real algebraic subvariety of $\R^d$ given by
\begin{align}\label{Ahatvariety}
\hat{\cA}=\cZ(\cJ):=\{x\in \R^d:p(y)=0~{\rm for} ~p\in \cJ\}.
\end{align}
In the special case $\cA=\R[x_1,\dots,x_d]$ we can take $p_1=x_1,\dots,p_d=x_d$\, and get $\hat{\cA}= \R^d$.
\begin{dfn}
A {\rm preordering} of\, $\cA$ is a subset $\cT$ of $\cA$ such that
\begin{align*}
\cT\cdot \cT\subseteq \cT,~~\cT + \cT \subseteq \cT,~~ 1 \in \cT,~~
a^2 \cT  \in \cT~~{\rm for~all}~~a\in \cA.
\end{align*}
\end{dfn}
Let $\sum \cA^2$ denote  the set of  finite sums $\sum_i a_i^2$ of squares of elements $a_i\in \cA$.
 Since $\cA$ is commutative,   $\sum \cA^2$ is invariant under multiplication and hence  the smallest  preordering of $\cA$.

For a preordering  $\cT$  of $\cA$, we define
\begin{align*}
\cK(\cT)=\{ x\in \hat{\cA}:  f(x)\geq 0 ~~ {\rm for~all}~f\in \cT\}.
\end{align*}
If $\cI$ is an ideal of $\cA$, its zero set is defined by $\cZ(\cI)=\{x\in \hat{\cA}: f(x)=0~{\rm for}~f\in \cI\}$.

The main concepts are introduced in the following definition, see \cite{schm2003} or \cite{marshall}.
\begin{dfn}\label{definionsmpmp}
A preordering $\cT$ of\, $\cA$ has the\\
$\bullet$  \emph{moment property (MP)}\, if   each $\cT$-positive linear functional $L$ on $\cA$ is a moment functional, that is, there exists a positive Borel measure $\mu\in \cM(\hat{\cA})$ such that
\begin{align}\label{Lrepreborelmu}
L(f) =\int_{\hat{\cA}}\, f(x) \, d\mu(x) \quad {\rm for~all}~~f\in \cA,
\end{align}
$\bullet$  \emph{strong moment property (SMP)} if  each $\cT$-positive linear functional $L$ on $\cA$ is a $\cK(\cT)$--moment functional, that is, there is a positive Borel measure $\mu\in \cM(\hat{\cA})$ such that\,
${\rm supp}\, \mu\subseteq \cK(\cT)$\, and
(\ref{Lrepreborelmu}) holds.
\end{dfn}

To state our main result (Theorem \ref{fibreth1})  we need some preparations.

Suppose that $\cT$ a finitely generated preordering of $\cA$ and
let ${\sf f}=\{f_1,\dots,f_k\}$ be a  sequence of generators of  $\cT$.
We consider a fixed $m$-tuple\, ${\sf h}=(h_1,\dots,h_m)$\,  of  elements $h_k\in \cA$.  Let  ${\sf h}(\cK(\cT))$  denote the subset of $\R^m$ defined by
\begin{align}\label{defkKT}
 {\sf h}(\cK(\cT))=\{(h_1(x),\dots,h_m(x)); x \in \cK(\cT)\}.
\end{align}
For $\lambda=(\lambda_1,\dots,\lambda_m)\in \R^m $ let $\cK(\cT)_\lambda$ be the  subset of $\hat{\cA}$ given by
$$
\cK(\cT)_\lambda=\{ x\in \cK(\cT): h_1(x)=\lambda_1,\dots,h_m(x)=\lambda_m \}
$$
and\, $\cT_\lambda$ the  preordering of $\cA$ generated by the sequence
$${\sf f}(\lambda):= \{f_1,\dots,f_k,h_1-\lambda_1 {\cdot} 1,\lambda_1{\cdot} 1-h_1{\cdot} 1,\dots,h_m-\lambda_m{\cdot} 1,\lambda_m {\cdot} 1
-h_m\}$$
Then $\cK(\cT)$ is  the disjoint union of fibre sets $\cK(\cT)_\lambda=\cK(\cT_\lambda)$, where $\lambda \in{\sf h}(\cK(\cT))$.

Let $\cI_\lambda$ be the ideal of $\cA$ generated by $h_1-\lambda_1{\cdot} 1,\dots,h_m-\lambda_m{\cdot} 1$. Then we have
$\cT_\lambda:=\cT +\cI_\lambda$ and  the preordering $\cT_\lambda/\cI_\lambda$ of the quotient algebra $\cA/\cI_\lambda$  is generated by
$$
\pi_\lambda({\sf f}):=\{\pi_\lambda(f_1),\dots,\pi_\lambda(f_k)\},$$
where $\pi_\lambda:\cA\to \cA/\cI_\lambda$ denotes the canonical map.

Further, let $\hat{\cI}_\lambda:=\cI(\cZ(\cI_\lambda))$ denote the ideal of all elements $f\in \cA$ which vanish on the zero set $\cZ(\cI_\lambda)$ of $\cI_\lambda$. Clearly, $\cI_\lambda \subseteq \hat{\cI}_\lambda$ and $\cZ(\cI_\lambda)=\cZ(\hat{\cI}_\lambda)$.
Set $\hat{\cT}_\lambda:=\cT +\hat{\cI}_\lambda$. Then $\hat{\cT}_\lambda/\hat{\cI}_\lambda$ is a preordering of the quotient algebra $\cA/\hat{\cI}_\lambda$.

Note that in  general we have $\cI_\lambda \neq \hat{\cI}_\lambda$ and equality holds if and only if the ideal $\cI_\lambda$ is {\it real}. The latter means that $\sum_j a_j^2\in \cI_\lambda$ for finitely  many elements $a_j\in \cA$ always implies that $a_j\in \cI_\lambda$ for all $j$.
\begin{thm}\label{fibreth1}
Let $\cA$ be a finitely generated commutative real unital algebra and let $\cT$ be a finitely generated preordering of  $\cA$. Suppose that  $h_1,\dots,h_m$ are elements of $\cA$ that are bounded on the  set  $\cK(\cT)$. Then the following are equivalent:\\
(i) $\cT$ has property (SMP) (resp. (MP)) in $\cA$.\\
(ii)  $\cT_\lambda$ satisfies (SMP) (resp. (MP)) in\, $\cA$\, for all\, $\lambda \in {\sf h}(\cK(\cT)).$\\
$(ii)^\prime$  $\hat{\cT}_\lambda$ satisfies (SMP) (resp. (MP)) in\, $\cA$\, for all\, $\lambda \in {\sf h}(\cK(\cT)).$\\
(iii) $\cT_\lambda/\cI_\lambda$ has (SMP) (resp. (MP)) in\, $\cA/\cI_\lambda$\, for all\, $\lambda \in {\sf h}(\cK(\cT)).$\\
$(iii)^\prime$   $\hat{\cT}_\lambda/\hat{\cI}_\lambda$ has (SMP) (resp. (MP)) in\, $\cA/\hat{\cI}_\lambda$\, for all\, $\lambda \in {\sf h}(\cK(\cT)).$
\end{thm}

 Proposition \ref{quotientsmpmp}(i) gives
the  implication (i)$\to$(ii),  Proposition \ref{quotientsmpmp}(ii) yields   the equivalences (ii)$\leftrightarrow$(iii) and (ii)$^\prime\leftrightarrow$(iii)$^\prime$, and  Proposition \ref{quotientsmpmp}(iii) implies equivalence (ii)$\leftrightarrow$(ii)$^\prime$.

The  main assertion of Theorem \ref{fibreth1} is the implication (ii)$\to$(i). Its proof is lengthy and technically involved. In the proof given below we reduce the general case to the case $\R_d[\ux]$.

\begin{prop}\label{quotientsmpmp}
Let $\cI$ be an ideal and $\cT$  a finitely generated preordering  of $\cA$. Let $\hat{\cI}$ be the ideal of all $f\in \cA$ which vanish on the zero set $\cZ(\cI)$ of $\cI$.\\
(i) If\,  $\cT$ satisfies (SMP) (resp. (MP)) in $\cA$,  so does   $\cT+\cI$.\\
(ii) $\cT+\cI$ satisfies (SMP) (resp. (MP)) in $\cA$ if and only if
$(\cT+\cI)/\cI$ does in $\cA/\cI$.\\
(iii)  $\cT+\cI$ obeys (SMP) (resp. (MP)) in $\cA$ if and only if $\cT+\hat{\cI}$ does.
\end{prop}
\begin{proof}
(i): See e.g. \cite{scheiderer}, Proposition 4.8.

(ii) See e.g. \cite{scheiderer},  Lemma 4.7.

(iii): It suffices to show that both preorderings $\cT+\cI$ and $\cT+\hat{\cI}$ have the same positive characters and the same positive linear functionals on $\cA$. For the sets of characters, using the equality $\cZ(\cI)=\cZ(\hat{\cI})$ we obtain
\begin{align*}
\cK(\cT+\cI)=\cK(\cT)\cap \cZ(\cI)=\cK(\cT)\cap \cZ(\hat{\cI})=\cK(\cT+\hat{\cI}).
\end{align*}
Since $\cT+\cI\subseteq \cT+\hat{\cI}$, a $(\cT+\hat{\cI})$-positive functional is trivially $(\cT+\cI)$-positive.

Conversely, let $L$ be a\, $(\cT+\cI)$-positive linear functional on $\cA$. Let  $f\in \hat{\cI}$ and take a polynomial $\tilde{f}\in \R_d[\ux]$ such that $\pi(\tilde{f})=f$.   Recall that  $\cA\cong \R_d[\ux]/ \cJ$ and $\cZ(\cI)\subseteq \hat{\cA}=\cZ(\cJ)$ by (\ref{Ahatvariety}). Then $\tilde{\cI}:=\pi^{-1}(\cI)$ is an ideal of $\R_d[\ux]$.

Let $x\in \cZ(\tilde{\cI}) (\subseteq \R^d).$ Clearly, $\cJ \subseteq \tilde{\cI}$ and hence  $\cZ(\tilde{\cI})\subseteq \cZ(\cJ)=\hat{\cA}$. Let $g\in \cI$ and choose  $\tilde{g}\in \tilde{\cI}$ such that $g=\pi(\tilde{g})$. Since $x$ annhilates $\cJ$ and $ x\in \cZ(\tilde{\cI}),$  we have $g(x)=\tilde{g}(x)=0.$ That is, $x\in \cZ(\cI)$. From the equality $\cZ(\cI)=\cZ(\hat{\cI})$  we obtain  $g(x)=0$ for $g\in \hat{\cI}$. Thus, since $f\in \hat{\cI},$  it follows that $f(x)=\tilde{f}(x)=0$. That is,
the polynomial $\tilde{f}$ vanishes on $\cZ(\tilde{\cI})$.

Conversely, let $L$ be a\, $(\cT+\cI)$-positive linear functional on $\cA$.   Recall that  $\cA\cong \R_d[\ux]/ \cJ$ and $ \hat{\cA}=\cZ(\cJ)$ by (\ref{Ahatvariety}). Then $\tilde{\cI}:=\pi^{-1}(\cI)$ is an ideal of $\R_d[\ux]$.

We verify that $\cZ(\tilde{\cI})\subseteq \cZ(\hat{\cI}).$
Let $x\in \cZ(\tilde{\cI}) (\subseteq \R^d).$ Clearly, $\cJ \subseteq \tilde{\cI}$ and hence  $\cZ(\tilde{\cI})\subseteq \cZ(\cJ)=\hat{\cA}$. Let $g\in \cI$ and choose  $\tilde{g}\in \tilde{\cI}$ such that $g=\pi(\tilde{g})$. Since $x$ annhilates $\cJ$ and $ x\in \cZ(\tilde{\cI}),$  we have $g(x)=\tilde{g}(x)=0.$ That is, $x\in \cZ(\cI)=\cZ(\hat{\cI}).$

Now let  $f\in \hat{\cI}$. We choose  $\tilde{f}\in \R_d[\ux]$ such that $\pi(\tilde{f})=f$.  Then $f(x)=\tilde{f}(x)$ for
$x\in \cZ(\cJ)=\hat{\cA}$. Hence, since $f$ vanishes on  $\cZ(\hat{\cI})$ and $\cZ(\tilde{\cI})\subseteq \cZ(\hat{\cI}),$
the polynomial $\tilde{f}$ vanishes on $\cZ(\tilde{\cI})$.
Therefore,
by the real Nullstellensatz \cite[Theorem 2.2.1, (3)]{marshall}, there are $m\in \N$ and $g\in \sum \R_d[\ux]^2$ such that $p:=(\tilde{f})^{2m}+g\in \tilde{\cI}$. Upon multiplying $p$ by some  even power of $\tilde{f}$ we can assume that $2m=2^k$ for some $k\in\N$. Then
\begin{align}\label{cprimepib}
\pi(p)=f^{2^k}+\pi(g)\in \cI ,\quad {\rm  where}\quad \pi(g)\in \sum \cA^2.
\end{align}
Being $(\cT+\cI)$-positive, $L$  annihilates $\cI$ and  is nonnegative on $\sum \cA^2$. Hence, by (\ref{cprimepib}), we obtain
$$
0=L(p)=L\big(f^{2^k}\big)+ L(\pi(g)),~~ L(\pi(g))\geq 0,~~ L\big(f^{2^k}\big)\geq 0.
$$
This implies that $L(f^{2^k})=0$.
Since $L$ is nonnegative on $\sum \cA^2$, the Cauchy-Schwarz inequality holds. By a repeated application of this inequality we derive
\begin{align*}
|L(f)|^{2^k}&\leq L(f^2)^{2^{k-1}} L(1)^{2^{k-1}}\leq L(f^4)^{2^{k-2}} L(1)^{2^{k-2}+2^{k-1}}\leq \dots \\&\leq L(f^{2^k}) L(1)^{1+\dots+2^{k-1}} =0.
\end{align*}
Therefore, $L(f)=0$.  That is, $L$ annihilates $\hat{\cI}$. Hence $L$ is $(\cT+\hat{\cI})$-positive. This  completes the proof of (iii).
\end{proof}
\noindent
{\it Proof of the implication (ii)$\to$(i) of Theorem \ref{fibreth1}:}\\
As noted at the beginning of this section,  $\cA$ is (isomorphic to) the quotient algebra\, $\R_d[\ux]/ \cJ$ and  $\hat{\cA}=\cZ(\cJ)$ by (\ref{Ahatvariety}).
 We  choose polynomials $\tilde{h}_j\in \R_d[\ux]$ such that $\pi(\tilde{h}_j)=h_j$. Then $\tilde{\cT}:=\pi^{-1}(\cT)$ is a  preordering of $\R_d[\ux]$.
Since  $\cJ\subseteq \tilde{\cT}$, each $\tilde{\cT}$-positive character  of $\R_d[\ux]$ annihilates $\cJ$, so it belongs to $\hat{\cA}=\cZ(\cJ)$. This implies that $\cK(\tilde{\cT})=\cK(\cT)$ and  $\tilde{h}_j(x)=h_j(x)$ for $x\in \cK(\cT),$ so that $ {\sf \tilde{ h}}(\cK(\tilde{\cT}))= {\sf h}(\cK(\cT)).$ Since $h_j$ is bounded on $\cK(\cT)$ by (ii), so is\, $\tilde{h}_j$\, on $\cK(\tilde{\cT})$.
 For  $\lambda \in {\sf h}(\cK(\cT)),$  $\tilde{\cI}_\lambda:=\pi^{-1}(\cI_\lambda)$ is an ideal of $\R_d[\ux]$. It is easily verified that
 $(\tilde{\cT})_\lambda= \tilde{\cT}+ \tilde{\cI}_\lambda.$

We will apply  Proposition \ref{quotientsmpmp}(ii) twice to  $\cA=\R_d[\ux]/\cJ$, that is, with $\cA$ replaced by $\R_d[\ux]$ and $\cI$ by $\cJ$ in Proposition \ref{quotientsmpmp}(ii).
Because $\cT_\lambda=((\tilde{\cT})_\lambda +\cJ)/\cJ$ has (SMP) (resp. (MP)) in $\cA=\R_d[\ux]/\cJ$ by assumption (ii), so does $(\tilde{\cT})_\lambda$ in $\R_d[\ux]$ by the if direction of Proposition \ref{quotientsmpmp}(ii). For the algebra $\R_d[\ux]$ the implication (ii)$\to$(i)  of Theorem \ref{fibreth1} was proved in \cite{schm2003}.  Since $(\tilde{\cT})_\lambda= \tilde{\cT}+ \tilde{\cI}_\lambda$ and $\tilde{h}_j$\, is bounded on $\cK(\tilde{\cT})$, this result applies to $\tilde{\cT}$. Hence  $\tilde{\cT}$ obeys (SMP) (resp. (MP)) in $\R_d[\ux]$ and  so does  $\cT=(\tilde{\cT} +\cJ)/\cJ$  in $\cA=\R_d[\ux]/\cJ$ by the only if direction of Proposition \ref{quotientsmpmp}(ii). $\hfill \Box$

\medskip

We state the special case $\cT=\sum \cA^2$ of Proposition \ref{quotientsmpmp}(ii) separately as
\begin{cor}\label{corquotient}
If\, $\cI$ is a ideal of $\cA$, then $\cI+\sum \cA^2$ obeys (MP) (resp. (SMP)) in $\cA$ if and only if\,  $\sum (\cA/\cI)^2$ does in $\cA/\cI$.
\end{cor}
Our applications to the rational moment problem in Section \ref{rationalmp} are based on  the following corollary.

\begin{cor}\label{corquotient2}
Let us retain the notation of Theorem \ref{fibreth1}. Suppose that for
 each\, $\lambda \in {\sf h}(\cK(\cT))$ there exist a finitely generated real unital algebra $\cB_\lambda$ and a surjective algebra homomorphism $\rho_\lambda:\cB_\lambda\to \cA/\cI_\lambda$ (or\, $\rho_\lambda:\cB_\lambda\to \cA/\hat{\cI}_\lambda$) such that  $\sum \cB_\lambda^2 $ obeys (MP) in $\cB_\lambda.$
Then $\cT$ has (MP) in $\cA$.
\end{cor}
\begin{proof}
Let $\lambda \in {\sf h}(\cK(\cT))$. We denote by
$\cJ^\lambda$  the kernel of the homomorphism $\rho_\lambda:\cB_\lambda\to \cA/\cI_\lambda$. Since
$\sum \cB_\lambda^2$ has (MP) in $\cB_\lambda$, it is obvious that  $\cJ^\lambda+ \sum \cB_\lambda^2$ has (MP) in $\cB_\lambda$. Therefore $\sum  (\cB_\lambda /\cJ^\lambda)^2$ satisfies (MP) in  $\cB_\lambda /\cJ^\lambda$ by Corollary \ref{corquotient}. Since the algebra homomorphism $\rho_\lambda$ is surjective,  the  algebra $\cB_\lambda /\cJ^\lambda$ is isomorphic to $\cA/\cI_\lambda$. Hence
$\sum (\cA/\cI_\lambda)^2$ has (MP) in $\cA/\cI_\lambda$ as well. Consequently, since $ \sum (\cA/\cI_\lambda)^2\subseteq \cT/\cI_\lambda  $, $\cT/\cI_\lambda$  obeys (MP) in $\cA/\cI_\lambda$. Then  $\cT$ has (MP)  by Theorem \ref{fibreth1},(iii)$\to$(i).

The proof under the assumption   $\rho_\lambda:\cB_\lambda\to \cA/\hat{\cI}_\lambda$ is almost the same; in this case the implication  (iii)$^\prime\to$(i) of Theorem \ref{fibreth1} is used.
\end{proof}
Theorem \ref{fibreth1} is  formulated for
a commutative unital {\it real} algebra $\cA$. In Sections \ref{applcmp} and \ref{twosdidecmp} we are concerned with commutative {\it complex} semigroup $*$-algebras.
This case can be easily reduced to Theorem \ref{fibreth1} as we discuss in what follows.

If $\cB$ is a  commutative complex $*$-algebra, its hermitian part $$\cA=\cB_h:=\{ b\in \cB:b=b^*\}$$ is a commutative real algebra.

Conversely, suppose that $\cA$ is a commutative real algebra. Then
its complexification $\cB:=\cA+ i\cA$  is a commutative  complex $*$-algebra with involution $(a_1+i a_2)^*:=a_1-i a_2$ and  scalar multiplication $$(\alpha+i\beta)(a_1+i a_2):= \alpha a_1-\beta a_2 + i (\alpha a_2 +\beta a_1),~~  \alpha,\beta \in \R,~a_1,a_2\in \cA,$$  and  $\cA$ is the hermitian part $\cB_h$ of $\cB$.  Let $b\in \cB$. Then we have $b=a_1+i a_2$ with $a_1,a_2\in \cA$ and since $\cA$  is commutative, we get
\begin{align}\label{sumba^2}
 b^*b=(a_1-i a_2)(a_1+i a_2)=a_1^2+a_2^2 +i(a_1a_2-a_2a_1) =a_1^2+a_2^2.
\end{align}
Hence, if $\cT$ is a preordering of $\cA$, then\, $b^*b\,\cT\subseteq \cT$\, for $b\in \cB$. In particular,  $\sum \cB^2=\sum \cA^2$.

Further, each $\R$-linear functional $L$ on $\cA$ has a unique extension $\tilde{L}$ to a $\C$-linear functional  on $\cB$. By (\ref{sumba^2}), $L$ is nonnegative on $\sum\cA^2$ if and only of  $\tilde{L}$ is nonnegative on $\sum \cB^2$, that is, $\tilde{L}$ is a positive linear functional on $\cB$.

\section{Rational moment problems in $\R^d$}\label{rationalmp}
We begin with some notation and some preliminaries to state the main results. For $q\in \R_d[\ux]$ and a subset $\cD\subseteq \R_d[\ux]$ we put
$$
\cZ(q):=\{x\in \R^d:q(x)=0\},\quad \cZ_\cD:=\cup_{q\in \cD}\, \cZ(q).
$$
Let   ${\bf{D}}(\R_d[\ux])$ denote the family of  all multiplicative subsets $\cD$  of $\R_d[\ux]$ such that $1\in \cD$ and $0\notin\cD$.
The real polynomials in a single variable $y$ are denoted by $\R[y]$.

Let $\{f_1,\dots,f_k\}$ be a  $k$-tuple of polynomials of $\R_d[\ux]$
and  $\cD \in{\bf{D}}(\R_d[\ux])$. Then $$\cA:=\cD^{-1}\R_d[\ux]$$ is a real  unital algebra which contains $\R_d[\ux]$ as a subalgebra.
Let   $\cT$ be the preordering of $\cA$ generated by $f_1,\dots,f_k$.
Further, we fix an  $m$-tuple\, ${\sf h}=\{h_1,\dots,h_m\}$\,  of  elements $h_j\in \cA$. For $\lambda \in \R^d$ let $\cI_\lambda$ be the ideal of $\cA$ generated by $h_j-\lambda_j\cdot 1$, $j=1,\dots,m.$ Recall that  ${\sf h}(\cK(\cT))$   is defined by (\ref{defkKT}).

\smallskip

We consider the following assumptions:
\smallskip

(i) {\it The functions $h_1,\dots,h_m\in \cA$ are bounded on the set ${\sf h}(\cK(\cT)$.}

(ii) {\it For all $\lambda \in {\sf h}(\cK(\cT))$ there are a finitely generated  set\, $\cE_{\lambda}  \in{\bf{D}}(\R[y])$,  a finite commutative unital real algebra $\cC_\lambda$,  and a surjective homomorphism
$$\rho_\lambda: \cE_{\lambda}^{-1}\R[y]\otimes \cC_\lambda\to \cD^{-1}\R[\ux]/\cI_\lambda\equiv\cA/\cI_\lambda.$$}
The following two theorems are the  main results of this section.
\begin{thm}\label{rationalmp2}
Suppose that the multiplicative set $\cD \in{\bf{D}}(\R_d[\ux])$ is finitely generated and assume  (i) and (ii).
Then the (finitely generated) real algebra $\cA$ obeys  (MP). For each $\cT$--positive linear functional $L$ on $\cA$ there is a positive regular Borel measure $\mu$ on $\hat{\cA}\cong \R^d\backslash\cZ_\cD$
such that $L(f)=\int_{\hat{\cA}} f \, d\mu$ for all $f\in \cA$.
\end{thm}
If the set $\cD$ is not finitely generated,  a number of  technical  difficulties appear: In general the algebra $\cA$ is no longer finitely generated and the character space\, $\hat{\cA}$\, is not locally  compact in the corresponding weak topology.  Recall that the fibre theorem  requires a finitely generated algebra, because it is based on the Krivine-Stengle Positivstellensatz. However, circumventing these technical problems  we have  the following general result concerning the multidimensional rational moment problem.
\begin{thm}\label{rationalmp1}
Let $\cD_0 \in{\bf{D}}(\R_d[\ux])$ and
let $\{f_1,\dots,f_k\}$ be a  $k$-tuple of polynomials of\, $\R_d[\ux].$
Suppose that for each finitely generated subset\, $\cD \in{\bf{D}}(\R_d[\ux])$\, of\, $\cD_0$ there exists  an $m$-tuple  ${\sf h}=\{h_1,\dots,h_m\}$\,  of  elements $h_j\in \cA:=\cD^{-1}\R_d[\ux]$  such that $\cD$,   $\cA$, and the preordering $\cT$ of $\cA$ generated by $f_1,\dots,f_k$ satisfy  assumptions (i) and (ii).

Let $\cT_0$ be the preordering of the algebra $\cA_0:=\cD_0^{-1}\R_d[\ux]$ generated by $f_1,\dots,f_k$.
Then for each $\cT_0$--positive linear functional $L_0$ on $\cA_0$ there exists a  regular positive Borel measure $\mu$ on $\R^d$ such that $f\in L^1(\R^d;\mu)$ and
\begin{align}\label{repL_0}
 L_0(f)=\int_{\R^d}  f d\mu\quad {\rm  for~ all}\quad f\in \cA_0.
\end{align}
\end{thm}
The proofs of these theorems require a number of  preparatory steps.
The following two results for $d=1$ are crucial and they are  of  interest in itself.
\begin{prop}\label{porprationad1}
Suppose that\, $\cE \in{\bf{D}}(\R[y])$. Then   for each positive linear functional $L$ on the real algebra  $\cB:=\cE^{-1}R[y]$ there exists a positive regular Borel measure $\mu$ on $\R$  such that $f\in L^1(\R,\mu)$ and
\begin{align}\label{intrep}
L(f)=\int f(x)\,d\mu(x)\quad {\rm for}\quad f\in \cB.
\end{align}
If $\cE$ is finitely generated, so is the algebra $\cB$ and\, $\sum \cB^2$ satisfies (MP) in $\cB$. In this case, $\cZ_\cE$ is a finite set and $\mu(\cZ_\cE)=0$.
\end{prop}
The proof of Proposition \ref{porprationad1} is given below. First we derive the following corollary.
\begin{cor}\label{corpropmp}
Suppose that $\cE\in{\bf{D}}(\R[y])$ is  finitely generated  and  $\cC$ is a finite commutative real unital algebra. Then\, $\sum\, (\cE^{-1}\R[y]\otimes \cC)^2 $\, obeys (MP) in the algebra\, $\cE^{-1}\R[y]\otimes \cC$.
\end{cor}
\begin{proof}
Throughout this proof we abbreviate $\cA:=\cE^{-1}\R[y]\otimes \cC$ and $\cB:=\cE^{-1}\R[y]$. Suppose that $L$ is a positive linear functional on the algebra $\cE^{-1}R[y]\otimes \cC$.

Let $\cI$ denote the ideal of  elements of
$ \cC$ which vanish on all characters of $\cC$.
Since $\cC$ is a finite algebra, the character set $\hat{\cC}$ is finite
and each positive linear functional on $\cC$ is a linear combination of characters, so it vanishes on $\cI$. Let $f\in\cB$. Clearly, $L(f^2\otimes \cdot)$ is a positive linear functional on $\cC$. Thus $\cC$ has a positive linear functional, $\hat{\cC}$ is not empty,  say $\hat{\cC}=\{\eta_1,\dots,\eta_n\}$ with $n\in \N$. Further,    $L(f^2\otimes c)=0$ for $c\in \cI$. Since the unital algebra   $\cC$ is spanned by its squares, we therefore have $L(g\otimes c)=0$ for all $g\in\cB$ and $c\in \cC$. We choose elements $e_1,\dots,e_n\in \cC$ such that $\eta_j(e_k)=\delta_{jk}$ for $j,k=1,\dots,n$ and define  linear functionals $L_k$ on the algebra $\cB$ by $L_k(\cdot )=L(\cdot \otimes e_k)$. Since $\eta_j(e_k^2-e_k)=0$ for all $j$, we conclude that $e_k^2-e_k\in \cI$. Hence
$$L_k(f^2)= L(f^2\otimes e_k)=L(f^2\otimes e_k^2)=L((f\otimes e_k)^2)\geq 0, \quad  f\in\cB.$$ That is, $L_k$ is a  positive linear functional on $\cB$. Therefore, by Proposition \ref{porprationad1}, there exists a positive regular Borel measure $\mu_k$ on $\R$ such that $L_k(f)=\int f\,d\mu_k$ for $f\in\cB$ and $\mu_k(\cZ_\cE)=0$.

 From Lemma \ref{hataddecription} below we obtain $\hat{\cB}=\{\chi_t; t\in \R\backslash \cZ_\cE\}$. It is easily  verified that $\hat{\cA}=\cup_{j=1}^n \{\chi_t\otimes \eta_j; t\in \R\backslash \cZ_\cE\}.$ We define a positive regular Borel measure $\mu$ on $\hat{\cA}$ by~ $\mu(\sum_j M_j\otimes \eta_j)=\sum_j \mu_j(M_j)$~ for Borel sets $M_j\subseteq\R\backslash \cZ_\cE$.

 Let $f\in \cB$ and $c\in \cC$. Since the element $c-\sum_j\eta_j(c)e_j$ is annihilated by all $\eta_k$,  it belongs to $ \cI$. Therefore,\,  $L(f\otimes c)=\sum_j \eta_j(c) L(f\otimes e_j)$ and hence
\begin{align*}
L(f\otimes c)&= \sum_{j=1}^n\, \eta_j(c) L(f\otimes e_j) =\sum_j\eta_j(c) \int_{\hat{\cB}} f \, d\mu_j\\&=\sum_j\eta_j(c) \int_{\hat{\cA}} (f\otimes e_j) \, d\mu=\int_{\hat{\cA}} (f\otimes c)\, d\mu,
\end{align*}
that is,  $L$ is given by the integral with respect to the measure $\mu$. This shows that  $\sum \cA^2$ obeys (MP) in the algebra $\cA$.
\end{proof}

In the proof of Proposition \ref{porprationad1} we use  the following  lemmas. We retain the notation established above.
\begin{lem}\label{auxlemmad1} Let $\cB$ be as in Proposition \ref{porprationad1} and  let $f=\frac{p}{q}\in \cB$, where  $q\in \cE$ and $p\in \R[y]$. Then $f(y)\geq 0$ for all  $y\in \R\backslash \cZ(q)$ if and only if
$f\in \sum \cB^2$.
\end{lem}
\begin{proof}
Suppose that $f\geq 0$ on $\R\backslash \cZ(q)$. Then  $q^2f=pq\geq 0$ on $\R\backslash \cZ(q)$ and hence on the whole real line, since $\cZ(q)$ is empty or finite. Since each nonnegative polynomial in one (!) variable  is a sum of two squares in $\R[y]$, we have  $pq=p_1^2+p_2^2$ with $p_1,p_2\in \R[y]$. Therefore, since $q\in \cE$, we get
$f=(\frac{p_1}{q})^2+(\frac{p_1}{q})^2\in \sum \cB^2.$ The converse implication is obvious.
\end{proof}

\begin{lem}\label{hataddecription}
 For $t\in \R^d\backslash \cZ_\cD$,   $p\in \R_d[\ux]$, and $q\in \cD$,  we define  $\chi_t(\frac{p}{q}) = \frac{p(t)}{q(t)}$. Then $\hat{\cA}=\{\chi_t; t\in \R^d\backslash \cZ_\cD\}$.
\end{lem}
\begin{proof}
First we note that for any  $t\in \R^d\backslash \cZ_\cD$, $\chi_t$ is a well-defined character on $\cA$, that is, $\chi_t\in \hat{\cA}$.

Conversely, suppose that $\chi\in \hat{\cA}$. Put $t_j=\chi(x_j)$ for $j=1,\dots,d$. Then $t=(t_1,\dots,t_d)\in \R^d$ and  $\chi(p(x))= p(\chi(x_1),\dots\chi(x_d)))=p(t)$ for $p\in \R_d[\ux].$ For $q\in \cD$ we have $1=\chi(1)=\chi(qq^{-1})=\chi(q)\chi(q^{-1})=q(t)\chi(q^{-1}).$ Hence $q(t)\neq 0$ for all $q\in \cD$,  that is, $t\in\R^d\backslash \cZ_\cD$. Further, $\chi(q^{-1})=q(t)^{-1}$ and therefore $\chi(\frac{p}{q})=\chi(p)\chi(q^{-1})=p(t)q(t)^{-1}=\chi_t(\frac{p}{q})$. Thus $\chi=\chi_t$.
\end{proof}
Set  $E:=\cA+C_c(\R^d).$ Let
$f=\frac{p}{q}\in \cA$, where
$p\in \R_d[\ux]$, $q\in \cD$, and $\varphi\in C_c(\R^d)$. Then $f+\varphi\in E$ and we define\, \begin{align}\label{defegeq}
f+\varphi \geq 0\quad {\rm if}\quad f(x)+\varphi(x)\geq 0\quad  {\rm for}\quad x\in\R^d\backslash\cZ(q).
\end{align}
Since $\cZ(q)$ is nowhere dense in $\R^d$,  $(E,\geq )$ is a real ordered vector space.

In the proofs of the following two lemmas  we  modify some arguments from Choquet's approach to the moment problem based on adapted spaces, see \cite{choquet} or \cite{berg}.
\begin{lem}\label{poslinfunct}
For each  positive linear functional $L$ on the ordered vector space $(E, \geq)$  there is a regular positive Borel measure $\mu$ on $\R^d$ such that
\begin{align}\label{repLint}
L(f
)=\int f\, d\mu \quad {for}\quad f\in\cA.
\end{align}
\end{lem}
\begin{proof}
Fix $p\in \R_d[\ux]$ and $q\in \cD$. Put $g=\frac{p^2+1}{q^2}$. We define $g$ on the whole space $\R^d$ by setting  $g=+\infty$ for $x\in \cZ(q)$ and abbreviate $\|x\|^2=x_1^2+\cdots+x_d^2$. Let $\varepsilon >0$ be arbitrary.  Obviously, the set
 $$K_\varepsilon:=\{x\in \R^d:\|x\|^2+1 +g \leq \varepsilon^{-1}\}$$ is compact and $q^2(x)\geq\varepsilon$ for $x\in K_\varepsilon$. Hence the compact set $K_\varepsilon$ and the closed set $\cU(q):=\{x\in \R^d:q^2(x)\leq \varepsilon/2 \}$ are disjoint, so by Urysohn's lemma there exists  a function $\eta_\varepsilon\in C_c(\R^d)$ such that $\eta_\varepsilon =1$ on $K_\varepsilon$,  $\eta_\varepsilon =0$ on $\cU(q)$ and $0\leq \eta_\varepsilon\leq 1$ on $\R^d$. Then we have $g\eta_\varepsilon \in C_c(\R^d)$ and
\begin{align}\label{ginequlai} g(x)\leq g(x)\eta_\varepsilon(x)  + \varepsilon[(\|x\|^2+1)g(x)+g^2(x)]\quad {\rm for}\quad x\in\R^d\backslash \cZ(q).
\end{align}
(Indeed, by the definitions of $\eta_\varepsilon$ and $K_\varepsilon$  we have $g(x)=g(x)\eta_\varepsilon(x)$ if $x\in K_\varepsilon$ and $g(x)<\varepsilon[(\|x\|^2+1)g(x)+g^2(x)]$ if $x\notin K_\varepsilon.$)

Since the restriction of $L$ to $C_c(\R^d)$ is a positive linear functional on  $C_c(\R^d)$,  by  Riesz' theorem there exists a positive regular Borel measure $\mu$ on $\R^d$ such that
\begin{align}\label{repvarphi}
L(\varphi)=\int \varphi\, d\mu\quad {\rm  for}\quad \varphi\in C_c(\R).
\end{align}
Using that $L$ is a positive functional with respect to the order relation  $\geq $ on $E$ and inequality (\ref{ginequlai}) and applying (\ref{repvarphi}) with $\varphi=g\eta_\varepsilon$ we derive
\begin{align*}
L(g)&\leq L(g\eta_\varepsilon)+\varepsilon L(\big(\|x\|^2+1)g+g^2\big)\\&=\int g\eta_\varepsilon d\mu+ \varepsilon L\big((\|x\|^2+1)g+g^2\big)\leq \int g d\mu+ \varepsilon L\big((\|x\|^2+1)g+g^2\big).
\end{align*}
Since $\varepsilon>0$ was arbitrary, we conclude that $L(g)\leq \int g d\mu$.

To prove the converse inequality, let $\cU_q$ be the set of $\eta\in C_c(\R^d)$ such that $0\leq \eta \leq 1$ on $\R^d$ and $\eta$ vanishes in some neighborhood of $\cZ(q)$. Since $\cZ(q)$ is nowhere dense in $\R^d$ and using  again by the positivity of $L$ and  (\ref{repvarphi}) we obtain
$$
\int g d\mu=\sup_{\eta\in \cU_q}~ \int g\eta\, d\mu =\sup_{\eta\in \cU_q}~ L(g\eta) \leq L(g).
 $$
Therefore, $L(g)=\int g d\mu$.

Thus we have proved the equality in (\ref{repLint})  for elements of the form $g=\frac{p^2+1}{q^2}$. Setting $p=0$, (\ref{repLint}) holds for all elements $\frac{1}{q^2}$, where $q\in \cD$, and  hence for  $\frac{p^2}{q^2}$ by linearity. Since $\cA$ is spanned by its squares, the equality (\ref{repLint}) holds
for all $f\in \cA$.
\end{proof}
\begin{lem}\label{vaguelimit}
Let $(\mu_i)_ {i\in I}$ be a net of positive regular Borel measures on $\R^d$ which converges vaguely   to a  positive regular Borel measure $\mu$ on $\R^d$. Let  $L$ is a linear functional on $\cA$ such that  $L(f
)=\int f\, d\mu_i$ for $i\in I$ and $f\in \cA$. Then $L(f)=\int fd\mu$ for $f\in \cA$.
\end{lem}
\begin{proof}
Let $f\in E$, $f\geq 0$ and $f=\frac{p}{q}$, where $q\in \cD$. Let $\cU_q$ be as in the preceding proof.  For $\eta\in \cU_q$  we have $f\eta \in C_c(\R^d)$ and hence $\lim_i \int f\eta\, d\mu_i= \int f\eta \, d\mu$.  Then
\begin{align}\label{intlh}
\int f\, d\mu =\sup_{\eta\in \cU_q} \int f\eta~ d\mu = \sup_{\eta\in \cU_q}\, \lim_i \int f\eta d\mu_i\leq  \lim_i \int f d\mu_i=L(f),
\end{align}
that is, $f\in L^1(\mu)$. Since $E$ is spanned by its squares,  $E\subseteq L^1(X;\mu).$

 Now we use  notation and facts from this proof of Lemma \ref{poslinfunct} and take an element
 $g=\frac{p^2+1}{q^2}$ as thererin
 Setting $h=(\|x\|^2+1)g+g^2$, inequality (\ref{ginequlai}) means that
\begin{align*}
g\leq g\eta_\varepsilon +\varepsilon h.
\end{align*}
Using this inequality  and  (\ref{intlh}), first with $f=g$ and then with $f=h$, we derive
\begin{align*}
\bigg|L(g)&-\int g\, d\mu\bigg|
=L(g)-\int g\, d\mu
=\int g\, d\mu_i -\int g\, d\mu\\ &
=\int (g-g\eta_\varepsilon)\, d\mu_i -\int (g-g\eta_\varepsilon)\, d\mu +
\int g\eta_\varepsilon  \, d\mu_i-\int g \eta_\varepsilon\, d\mu\\ &\leq \varepsilon\bigg( \int h\, d\mu_i + \int h d\mu\bigg)+\int g \eta_\varepsilon\, d\mu_i-\int g\eta_\varepsilon\, d\mu\\ &\leq \varepsilon( L(h) + L(h))+\int g\eta_\varepsilon\, d\mu_i-\int g \eta_\varepsilon\, d\mu .
\end{align*}
Since $g\eta_\varepsilon\in C_c(\R^d)$, $\lim_i\, \int g\eta_\varepsilon\, d\mu_i=\int g\eta_\varepsilon\, d\mu$ by the vague convergence.   Therefore, taking  $ \lim_i$ in the preceding inequality yields $|L(g)-\int g\, d\mu|\leq 2\, \varepsilon L(h)$. Hence $L(g)=\int g\, d\mu$ by letting $\varepsilon\to +0$. Arguing as in the proof of Lemma \ref{poslinfunct} this implies that $L(f)=\int f\, d\mu$ for all $f\in \cA$.
\end{proof}

{\it Proof of Proposition \ref{porprationad1}.}
Set $E:=\cB+C_c(\R)$. Let $(E,\geq )$  be the real ordered vector space defined above, see (\ref{defegeq}). Let $f\in \cB$. By Lemma \ref{auxlemmad1}, we  have $f\geq 0$ if and only if $f\in \sum \cB^2$. Hence $L(f)\geq 0$.
Since $\R[y]\subseteq \cB$,  $\cA$ is a cofinal subspace of the ordered vector space $(E,\geq )$. Therefore, $L$ can be extended to a positive linear functional $\tilde{L}$ on $(E,\geq  )$. Applying Lemma \ref{poslinfunct}, with $d=1$ and $\cA$ replaced by $\cB$,  to the functional $\tilde{L}$ yields (\ref{defegeq}).

Suppose in addition that $\cE$ is finitely generated. Let $q_1,\dots,q_r$ be generators of $\cE$. Then the algebra $\cB$ is generated by $y,q_1^{-1},\dots,q_r^{-1}$, so $\cB$ is finitely generated. Further,  $\cZ_\cE=\cap_j \cZ(q_j)$, so the set $\cZ_\cE$ is finite. If $q\in \cE$, then $q^{-2}\in \cA$ and hence $L(q^{-2})=\int q^{-2} d\mu <\infty$. Therefore, $\mu(\cZ(q))=0$, so that  $\mu(\cZ_\cE)=0$. Hence, by Lemma \ref{hataddecription},  the integral in (\ref{defegeq}) is over the set $\hat{\cA}\cong \R\backslash \cZ_\cE$ and $\mu$ is a  positive regular Borel measure on\, $\hat{\cA}$. Thus, $ \sum \cB^2$\, has property (MP). $\hfill \Box$
\medskip

{\bf Remarks.}
1. It is possible that there is {\it no} nontrivial positive linear functional  on the algebra $\cB$ in Proposition \ref{porprationad1}; for instance, this happens if $\cZ_\cE=\R$.

2.
Let $\cB$ be a real unital algebra  of rational functions in one variable and consider the following conditions:
\smallskip

(i)  All positive linear functionals on $\cB$ can be represented as integrals with respect to some positive regular Borel measure on $\R$ (or on the character set $\hat{\cB}$),
\\
(ii)  $f\in \cB$ and $\chi(f)\geq 0$ for all $\chi\in \hat{\cB}$ implies that $f\in \sum\cB^2$.
\smallskip

It seems to be of interest to characterize those algebras $\cB$ for which (i) or (ii) holds.
By Proposition \ref{porprationad1} and Lemma \ref{auxlemmad1}  the algebra $\cE^{-1}\R[y]$ satisfies (i) {\it and} (ii).
Since  $\cB:=\R[\frac{1}{y^2+1},\frac{1}{y^2+2}]$ is isomorphic to the polynomial algebra $\R[x_1,x_2]$,  (i) does not hold for $\cB$. For   $\cB:=\R[x,\frac{1}{x^2+1}]$ condition (i) is true, while (ii) is not fulfilled.

3. Another important open problem is the following: Characterize the finitely generated real unital algebras $\cA$ for which $\sum \cA^2$ has (MP).
\medskip

{\it Proof of Theorem \ref{rationalmp2}.}
First we note that the algebra $\cA$ is finitely generated, because  $\cD$ is finitely generated by assumption. Let us fix $\lambda \in {\sf h}(\cK(\cT))$
and abbreviate $\cB_\lambda:= \cE_\lambda^{-1}\R[y]\otimes \cC_\lambda.$
Then, by Corollary \ref{corpropmp},\,  $\sum \cB_\lambda^2$ has (MP) in the algebra\, $\cB_\lambda$. Therefore, Corollary  \ref{corquotient2} applies and shows that $\cT$ satisfies (MP) in $\cA$. The description of $\hat{\cA}$ was given in Lemma \ref{hataddecription}. $\hfill \Box$

\medskip

{\it Proof of Theorem \ref{rationalmp1}.}
 Let $I$ denote the net of all finitely generated multiplicative subsets $\cD$ of $\R_d[\ux]$ such that  $\cD\subseteq \cD_0$. Fix $\cD\in I$. Since $\cD$ satisfies the assumptions (i) and (ii),    Theorem \ref{rationalmp2} applies to $\cD$ and $L_0\lceil \cA_\cD$. Hence there exists a regular positive Borel measure $\mu_\cD$ on $\R^d$ such that $L_0(f)=\int f\, d\mu_\cD$ for $f\in \cA_\cD$. Since $\int 1\, d\mu_\cD=\mu_F(\R^d)=L_0(1)$ for all $\cD\in I$, it follows that the set $\{\mu_\cD;\cD\in I\}$ is  relatively vaguely compact. Hence there is a subnet of the net $(\mu_\cD)_{\cD\in I}$ which converges vaguely to some regular positive Borel measure $\mu$ on $\R^d$. For notational simplicity let us assume that the net $(\mu_\cD)_{\cD\in I}$ itself has this property. Now we fix $\cF\in I$ and apply Lemma \ref{vaguelimit} to the net $(\mu_\cD)_{\cD\supseteq \cF}$, the algebra $\cA_{\cF}$ and the functional $L=L_0\lceil \cA_\cF$ to conclude that $L_0(f)= \int f\, d\mu$ for  $f\in \cA_\cF$. Since each function $f\in \cA_0$ is contained in some algebra $\cA_\cF$ with $\cF\in I$,  (\ref{repL_0}) is satisfied. This  completes the proof.
$\hfill$ $\Box$

\smallskip

The general fibre Theorem \ref{fibreth1} fits nicely to  the multidimensional rational moment problem,  because in general the corresponding algebra $\cA$
contains  more  {\it bounded}\,  functions  on the semi-algebraic set $\cK(\cT)$ than in the polynomial case.

We illustrate the use of our Theorems \ref{rationalmp2} and \ref{rationalmp1} by some examples.
\begin{exa}
First let $d=2$.  Suppose that  $\cD$ contains $x_1-\alpha $ and  the semi-algebraic set $\cK(\cT)$ is a subset of $\{ (x_1,x_2): |x_1-\alpha|\geq c\}$ for some $\alpha\in \R$ and $c>0$. Then $h_1:=(x_1-\alpha )^{-1}$ is in $\cA$ and bounded on $\cK(\cT)$, so assumption (i) holds. Let $\lambda\in h_1(\cK(\cT))$. Then $x_1=\lambda^{-1}+\alpha$ in $\cA_\lambda$ , so $\cA_\lambda$ consists of rational functions in $x_2$ with denominators from some set\, $\cE_\lambda \in{\bf{D}}(\R[x_2])$. Hence assumption (ii) is satisfied with $\cC_\lambda =\C$. Thus  Theorem \ref{rationalmp2} applies if\,  $\cD$ is finitely generated. Replacing $\cD$ by $\cD_0$ and $\cT$ by $\cT_0$ in the preceding, the assertion of
Theorem \ref{rationalmp1} holds as well.

The above setup extends at once to arbitrary $d\in \N, d\geq 2,$ if we assume that $x_j{-}\alpha_j\in \cD$ and $|x_j{-}\alpha_j|\geq c$ on  $ \cK(\cT)$ for some $\alpha_j\in \R$, $c>0$, and  $j=,\dots,d-1.$
\end{exa}
\begin{exa}
Again we take  $d=2$. Let $  \cE \in{\bf{D}}(\R[x_2])$ and let $p$ be a nonconstant polynomial from $R[x_1]$. Suppose that $\cD$ is generated by $\cE$ and\, $p^{2}+\alpha$ for some $\alpha>0$. Obviously, $h_1=(p^{2}+\alpha)^{-1}\in \cA$ is bounded on  each semi-algebraic set $\cK(\cT).$ Let $\lambda\in h_1(\cK(\cT))$. Then we have $p(x_1)^{2}=\lambda^{-1}+\alpha$ in the fibre algebra $\cA_\lambda$. Let $\cC$ be the quotient algebra of $\R[x_1]$ by the ideal generated by\, $p(x_1)^2-\lambda^{-1}-\alpha$. Then  $\cA_\lambda$ is generated by two subalgebras which are isomorphic to\, $\cC$ and  $\cE^{-1}\R[x_2]$, respectively. Hence assumptions (i) and (ii) are fulfilled, so $\sum \cA^2$ has (MP) by Theorem \ref{rationalmp2} if\, $\cD$ is finitely generated. In the general case  Theorem \ref{rationalmp1} applies.
\end{exa}
In the following two examples we suppose that the sets $\cD$ are finitely generated.
\begin{exa}\label{fibrealgebraR}
Suppose that $\cD\in{\bf{D}}(\R[\ux])$ contains  the polynomials $q_j=1+x_j^2$ for $j=1,\dots,d$. Let $\cT=\sum (\cA_\cD)^2$. Then $\cK(\cT)=\widehat{\cA_{\cD}}=\R^d$ by Lemma \ref{hataddecription}.  Setting $h_j=q_j(x)^{-1}$,  $h_{d+j}=x_jq_j(x)^{-1}$ for $j=1,\dots,d$, all  $h_l\in \cA_\cD$  are bounded on $\cK(\cT)$.

Let $\lambda \in {\sf h}(\cK(\cT)).$ Then $\lambda_j=q_j(\lambda)^{-1}\neq 0$ and $\lambda_{d+j}=\lambda_jq_j(\lambda)^{-1}$;  hence we have $x_j=\lambda_{d+j}\lambda_j^{-1}$ for $j=1,\dots,d$ in the fibre algebra $(\cA_\cD)_\lambda.$ Thus, $(\cA_\cD)_\lambda=\R.$ Taking $\cE_\lambda=\{1\}$, we have $\cE_\lambda^{-1}\R[\ux]=\R[\ux]$, so (i) and (ii) are obviously satisfied.

Therefore, $\sum (\cA_\cD)^2$\, obeys (MP) by Theorem \ref{rationalmp2}. That is,  each positive linear functional on the algebra $\cA_\cD$ is given by some positive measure on\, $\widehat{\cA_\cD}=\R^d$.

The same conclusion and almost  the same reasoning
are valid if we assume instead that $\cD$ contains the polynomial $q=1+x_1^2+\dots+x_d^2.$ In this case we set $h_j=x_jq(x)^{-1}$ for $j=1,\dots,d$ and $h_{d+1}=q(x)^{-1}.$
\end{exa}

\begin{exa}
Suppose now that $\cD\in{\bf{D}}(\R[\ux])$ contains only the polynomials $q_j=1+x_j^2$ for $j=1,\dots,d-1$. Let $\cT=\sum (\cA_\cD)^2$. Then  $h_j:=q_j(x)^{-1}$ and  $h_{d+j-1}:=x_jq_j(x)^{-1}$ for $j=1,\dots,d-1$   are in $\cA_\cD$ and  bounded on $\cK(\cT)=\widehat{\cA_\cD}$.
Arguing as in Example \ref{fibrealgebraR} it follows that  $x_j=\lambda_{d+j-1}\lambda_j^{-1}$ for $j=1,\dots,d-1$ in the algebra $(\cA_\cD)_\lambda.$  Therefore  $(\cA_\cD)_\lambda$ is an algebra $\cE_\lambda^{-1}\R[x_d]$  of rational functions in the variable $x_d$  for some finitely generated set\, $\cE_\lambda \in{\bf{D}}(\R[x_d])$.

Then, again by Theorem \ref{rationalmp2}, $\sum (\cA_\cD)^2$\, satisfies (MP). The same is true if we assume instead that the polynomial $q=1+x_1^2+\dots+x_{d-1}^2$ is in $\cD\in{\bf{D}}(\R[\ux])$.
\end{exa}

{\bf Remark.} Assumption (ii) is needed to ensure that the fibre preorderings $\cT_\lambda$ satisfy (MP). The crucial result for this is  Proposition \ref{porprationad1} which states  that the cone\, $\sum\, (\cE^{-1}\R[y])^2 $ obeys (MP) for finitely generated $\cE\in{\bf{D}}(\R[y])$. A similar assertion holds for several other algebras of rational functions as well; a sample is $\R[x_1,x_2,\frac{1}{x_2^2+1}]$.  Replacing $\cE^{-1}\R[y]$  by such an algebra  the fibre theorem  will lead to further results on the multidimensional rational moment problem.

\section{An extension theorem}\label{extensionsection}
In this section we derive a theorem which characterizes  moment functional on $\R^d$ in terms of extensions.

Throughout let $\cA$ denote the real algebra of functions on $(\R^d)^\times:=\R^d\backslash \{0\}$ generated by the polynomial algebra $\R_d[\ux]$ and the functions
\begin{align}\label{sumfkl1}
f_{kl}(x):=x_kx_l(x_1^2+\dots+x_d^2)^{-1},~~{\rm  where}~~ k,l=1,\dots,d,\, x\in\R^d\backslash \{0\} .
\end{align}
Clearly, these functions satisfy the identity
\begin{align}
\sum_{k,l=1}^d\, f_{kl}(x)^2=1.
\end{align}
That is, the functions $f_{kl}$, $k,l=1,\dots,d$, generate the coordinate algebra $C(S^{d-1})$ of the unit sphere $S^{d-1}$ in $\R^d$.
The next lemma describes  the character set $\hat{\cA}$  of $\cA$.
\begin{lem}\label{charactersetahat}
The set $\hat{\cA}$ is parameterized by the disjoint union of $\R^d\backslash \{0\} $ and $ S^{d-1}$. For\, $x\in \R^d\backslash \{0\}$\, the  character $\chi_x$  is the  evaluation of functions  at $x$ and  for $t\in S^{d-1}$ the character $\chi^t$ acts by $\chi^t(x_j)=0$ and $\chi^t(f_{kl})=f_{kl}(t)$, where $j,k,l=1,\dots,d$.
 \end{lem}
 \begin{proof}
It is obvious that for any $x\in \R^d\backslash \{0\}$ the point evaluation $\chi_x$ at $x$ is a character on the algebra $\cA$ satisfying $(\chi(x_1),\dots,\chi(x_d))\neq 0$.

Conversely, let $\chi$  be a character of $\cA$ such that
 $x:=(\chi(x_1),\dots,\chi(x_d))\neq 0$.
Then the identity $(x_1^2+\dots+x_d^2)f_{kl}=x_kx_l$ implies  that
$$(\chi(x_1)^2+\dots+\chi(x_d)^2)\chi(f_{kl})=\chi(x_k)\chi(x_l)$$ and therefore
$$\chi(f_{kl})=(\chi(x_1)^2+\dots+\chi(x_d)^2)^{-1}\chi(x_k)\chi(x_l)=f_{kl}(x).$$
Thus  $\chi$ acts on the generators $x_j$ and $f_{kl}$, hence on the whole algebra $\cA$, by point evaluation at $x$, that is, we have $\chi=\chi_x$.

Next let us note that the quotient of $\cA$ by the ideal generated by $\R_d[\ux]$ is (isomorphic to) the algebra $C(S^{d-1})$.
Therefore, if $\chi$ is a character of $\cA$ such that $(\chi(x_1),\dots,\chi(x_d))=0$, then it gives  a character on the algebra $C(S^{d-1})$. Clearly, each character of $C(S^{d-1})$ comes from a point  of $S^{d-1}$. Conversely, each point  $t\in S^{d-1}$ defines a unique character of $\cA$ by $\chi^t(f_{kl})=f_{kl}(t)$ and $\chi^t(x_j)=0$ for all $k,l,j$.
\end{proof}

 \begin{thm}\label{mpshperesinrd}
The preordering $\sum \cA^2$ of the algebra $\cA$ satisfies (MP), that is, for each positive linear functional $\cL$ on $\cA$ there exist positive Borel measures $\nu_0$ on $S^{d-1}$ and $\nu_1\in \cM(\R^d\backslash \{0\}) $ such that for all polynomials $g$ we have
 \begin{align}\label{descptionltilde}
 \cL&(g(x,f_{11}(x),\dots, f_{dd}(x)))\nonumber\\&=\int_{S^{d-1}} g(0,f_{11}(t),\dots,f_{dd}(t)) \, d\nu_0(t) +\int_{\R^d\backslash \{0\}} g(x,f_{11}(x),\dots,f_{dd}(x)) ~ d\nu_1(x).
 \end{align}
 \end{thm}
 \begin{proof}It suffices to prove that $\sum \cA^2$ has (MP). The  assertions follow  then from the definition of the property (MP) and the explicit form  of the character set  given in Lemma \ref{charactersetahat}.

 From the description of  $\hat{\cA}$ it is obvious that the functions $f_{kl}$, $k,l=1,\dots,d,$ are bounded on $\hat{\cA}$, so we can take them as functions $h_j$ in Theorem \ref{fibreth1}. Let us fix a  non-empty fire  for  $\lambda=(\lambda_{kl})$, where  $\lambda_{kl}\in\R$ for all $k,l$, and consider the quotient algebra\, $\cA/\cI_\lambda$\, of $\cA$ by the fibre ideal\, $\cI_\lambda.$

In the algebra $\cA/\cI_\lambda$ we  have $\chi(f_{kl})=\lambda_{kl}$ for $\chi\in \hat{\cA}$ and all $k,l$.  Let $\chi=\chi_x$, where  $x\in\R^d\backslash \{0\}$. Then $\chi_x(f_{kl})=f_{kl}(x)=\lambda_{kl}$. Since  $1=\sum_{k} f_{kk}(x)=\sum_k \lambda_{kk}$, there is a $k$ such that $\lambda_{kk}\neq 0$. From the equality   $\lambda_{kk}=f_{kk}(x)=x_k^2(x_1^2+\dots+x_d^2)^{-1}$ it follows that $x_k\neq 0$. Thus $\frac{\lambda_{kl}}{\lambda_{kk}}=\frac{f_{kl}(x)}{f_{kk}(x)}=\frac{x_l}{x_k}$, so that
\begin{align}\label{quetienalgrelation}
x_l= \lambda_{kl}\lambda_{kk}^{-1}\,x_k \quad {\rm for}~~ l=1\dots,d.
\end{align}
If $\chi=\chi^t$ for $t\in S^d$, then $\chi(x_l)=\chi(x_k)=0$, so  (\ref{quetienalgrelation}) holds trivially. That is, in the quotient algebra $\cA/\cI_\lambda$ we have the relations (\ref{quetienalgrelation}) and $ f_{kl}=\lambda_{kl}$.
 This implies that\, $\cA/\cI_\lambda$\,   is an algebra of polynomials in the single variable $x_k$. Hence it follows from Hamburger's theorem  that the preordering $\sum (\cA/\cI_\lambda)^2$ satisfies (MP) in $\cA/\cI_\lambda$. Therefore, by  Theorem \ref{fibreth1},(iii)$\to$(i),  $\cT=\sum \cA^2$ obeys (MP) in $\cA$.
\end{proof}

The main result of this  section is the following extension theorem.
\begin{thm}\label{extensionthm}
A linear functional $L$ on $\R_d[\ux]$ is a  moment functional if and only if it has an extension to a positive linear functional $\cL$ on the larger algebra $\cA$.
\end{thm}

\begin{proof}
Assume first that $L$ has an extension to a positive linear functional $\cL$ on $\cA$. By Theorem \ref{mpshperesinrd}, the functional $\cL$ on $\cA$ is of the form described by equation (\ref{descptionltilde}). We define a positive Borel measure $\mu$ on $\R^d$ by $$\mu(\{0\})=\nu_0(S^{d-1}),\quad  \mu(M\backslash \{0\})=\nu_1(M\backslash \{0\}).$$ Let $p\in \R_d[\ux]$. Setting $g(x,0,\dots,0)=p(x)$ in the equation of Theorem \ref{mpshperesinrd}, we get
$$
L(p)=\cL(p)= \nu_0(\{0\})p(0)+ \int_{\R^d\backslash \{0\}} g(x,0,\dots,0) ~ d\nu_1(x)=\int_{\R^d} p(x)\, d\mu(x).
$$
Thus $L$ is moment functional on $\R_d[\ux]$ with representing measure $\mu$.

Conversely, suppose that  $L$ is a moment functional on $\R_d[\ux]$ and let $\mu$ be a  representing measure.
Since $f_{kl}(t,0,\dots,0)=\delta_{k1}\delta_{l1}$ for $t\in \R, t\neq 0,$  we have
$\lim_{t\to 0} f_{kl}(t,0,\dots,0)=\delta_{k1}\delta_{l1}$. Hence there is  a well-defined character on the algebra $\cA$ given by
$$\chi(f)=\lim_{t\to 0} f(t,0,\dots,0),\quad f\in \cA,$$
and $\chi(p)=p(0)$ for $p\in \R[\ux]$. Then, for $f\in \R_d[\ux]$, we have

\begin{align}\label{defltilde}
L(f)=\mu(\{0\})\chi(f)+\int_{\R^d\backslash \{0\}} f(x)\, d\mu(x).
\end{align}

For $f\in \cA$ we define $\cL(f)$ by the right-hand side of (\ref{defltilde}). Then $\cL$ is a  positive linear functional on $\cA$ which extends $L$.
\end{proof}

Remarks. 1. The problem of  characterizing moment sequences in terms of  extensions have been studied in several papers such as \cite{stochelsz}, \cite{pv}, and \cite{stochelsz2}.

2. Another type of extension theorems has been  derived in \cite{pv}. The main difference  to the above theorem is that in \cite{pv}, see e.g. Theorem 2.5, a function $$h(x):=(1+x_1^2+\dots+x_d^2+p_1(x)^2+\dots+p_k(x)^2)^{-1}$$ is added to the algebra, where $p_1,\dots,p_k\in \R_d[\ux]$ are fixed. Then $h(x)$ is  bounded  on the  character set and  so are $x_jh$ and $x_jx_kh$ for $j,k=1,\dots,d$. The existence assertions of the results in \cite{pv} follow also  from  Theorem \ref{fibreth1}. Note that in this case the representing measure for the extended functional is unique (see \cite[Theorem 2.5]{pv}).

2. The measure $\nu_1$ in Theorem \ref{mpshperesinrd} and  the representing measure $\mu$ for the functional $\cL$ in Theorem \ref{extensionthm} are  not uniquely determined by $\cL$. (A counter-example can be easily constructed by taking an appropriate measure supported by a coordinate axis.)  Let $\mu_{\rm rad}$ denote the measure on $[0,+\infty)$  obtained by transporting $\mu$  by the mapping $x\to \|x\|^2$. Then, as shown in \cite[p. 2964, Nr 2.]{ps}, if  $\mu_{\rm rad}$ is determinate on $[0,+\infty)$, then $\mu$ is is uniquely determined by $\cL$. Thus,  Theorem \ref{extensionthm} fits nicely to the determinacy results via disintegration of measures developed in \cite[Section 8]{ps}.

\section{Application to the complex moment problem}\label{applcmp}
Given a  complex $2$-sequence $s=(s_{m,n})_{(m,n)\in \N_0^2}$ the  complex moment problem asks when does there exist a positive Borel measure $\mu$ on $\C$ such that
the function $z^m\ov{z}^n$ on $\C$ is $\mu$-integrable and
\begin{align}\label{cmp}
s_{mn} =\int_\C~ z^m\ov{z}^n\, d\mu(z)\quad {\rm for~all}\quad (m,n)\in \N_0^2 .
\end{align}

The semigroup $*$-algebra  $\C [\N_0^2]$ of the $*$-semigroup $\N_0^2$ with involution $(m,n):=(n,m)$, $(m,n)\in \N_0^2$, is the $*$-algebra $\C[z,\ov{z}]$ with involution given by $z^*=\ov{z}$. If $L$ denotes the linear functional on $\C[z,\ov{z}]$ defined by
\begin{align*}
L(z^m\ov{z}^n)=s_{m,n},\quad (m,n)\in \N_0^2
\end{align*} then (\ref{cmp}) means that
\begin{align*}
L_s(p)=\int_\C \, p(z,\ov{z})\, d\mu(z), \quad p\in\C[z,\ov{z}].
\end{align*}
Clearly,  $\N_0^2$ is a subsemigroup of the larger $*$-semigroup
$$
\cN_+=\{(m,n)\in \Z^2: m+n\geq 0\} \quad {\rm with ~involution}\quad (m,n)^*=(n,m).
$$
The following fundamental theorem was proved by J. Stochel and F.H. Szafraniec \cite{stochelsz}.
\begin{thm}\label{extcmpstochelsz}
A linear functional $L$ on $\C[z,\ov{z}]$ is a moment functional if and only if $L$ has an extension to a positive linear functional $\cL$ on the $*$-algebra $\C[\cN_+]$.
\end{thm}
In \cite{stochelsz} this theorem was stated
in terms of semigroups:\smallskip

{\it A complex sequence $s=(s_{m,n})_{(m,n)\in \N_0^2}$   is a moment sequence on $\N_0^2$ if and only if
there exists a positive semidefinite sequence\,  $\tilde{s}=(\tilde{s}_{m,n})_{(m,n)\in \cN_+}$\,  on the $*$-semigroup $\cN_+$ such that $\tilde{s}_{m,n}=s_{m,n}$ for all $(m,n)\in \N_0^2$.}
\medskip

In order to prove Theorem \ref{extcmpstochelsz} we first describe the semigroup $*$-algebra $\C[\cN_+]$. Clearly,   $\C[\cN_+]$ is the complex $*$-algebra   generated by the functions $z^m\ov{z}^n$ on $\C\backslash \{0\}$, where $m,n\in \Z$ and $m+n\geq 0$.
 If $r(z)$ denotes the modulus and $u(z)$ the phase of $z$, then  $z^m\ov{z}^n=r(z)^{m+n}u(z)^{m-n}$. Setting $k=m+n$, it follows that
 $$\C[\cN_+]={\Lin}\, \{ r(z)^k u(z)^{2m-k};k\in \N_0,m\in \Z\}\,.
 $$
 The functions $r(z)$ und $u(z)$ itself are not in $\C[\cN_+]$, but $r(z)u(z)=z$ and $v(z):=u(z)^2=z \ov{z}^{-1}$ are in $\C[\cN_+]$ and they  generate the $*$-algebra $\C[\cN_+]$. Writing $z=x_1+{\ii }x_2$ with $x_1,x_2\in \R$, we get
 $$
 1+v(z)=1+\frac{x_1+{\ii}x_2}{x_1-{\ii}x_2} =2\,\frac{x_1^2+{\ii}\,x_1x_2}{x_1^2+x_2^2},~~ 1-v(z)= 2\, \frac{x_2^2-{\ii}\,x_1x_2}{x_1^2+x_2^2}\,.
 $$
This implies that the complex algebra $\C[\cN_+]$ is generated  by the  five functions
\begin{align}\label{fivef}
x_1,~~ x_2,~~ \frac{x_1^2}{x_1^2+x_2^2}\,,~~\frac{x_2^2}{x_1^2+x_2^2}\,,~~\frac{x_1x_2}{x_1^2+x_2^2}\,.
\end{align}
Obviously,  the hermitian part $\C[\cN_+]_h$ of the complex $*$-algebra\, $\C[\cN_+]$\, is just the {\it real} algebra generated by the functions  (\ref{fivef}).
This real algebra is the special case $d=2$ of the $*$-algebra $\cA$ treated in Section \ref{extensionsection}. Therefore, if we identify $\C$ with $\R^2$, the assertion of Theorem \ref{extcmpstochelsz} follows at once from  Theorem \ref{extensionthm}.

\section{Application to the two-sided complex moment problem}\label{twosdidecmp}

The two-sided complex moment problem is the moment problem for the $*$-semigroup $\Z^2$ with involution $(m,n):=(n,m)$.
Given a   sequence $s=(s_{m,n})_{(m,n)\in \Z^2}$ it asks when does there exist a positive Borel measure $\mu$ on $\C^\times:=\C\backslash \{0\}$  such that
the function $z^m\ov{z}^n$ on $\C^\times$ is $\mu$-integrable and
\begin{align*}
s_{mn} =\int_{\C^\times}~ z^m\ov{z}^n\, d\mu(z)\quad {\rm for~all}\quad (m,n)\in \Z^2 .
\end{align*}
Note that this requires conditions for the measure $\mu$ at infinity and at zero.

The following basic result was obtained by T.M. Bisgaard \cite{bisgaard}.
\begin{thm}\label{bisgaardthm}
A linear functional $L$ on $\C[\Z^2]$ is a moment functional if and only if $L$ is a positive functional, that is, $L(f^*f)\geq 0$ for all $f\in \C[\Z^2]$.
\end{thm}
In terms of $*$-semigroups the main assertion of this theorem says that {\it each positive semidefinite sequence  on $\Z^2$ is a moment sequence on $\Z^2$.} This result  is somewhat surprising, since $C^\times$ has dimension $2$ and no additional condition (such as strong positivity or some appropriate extension) is required.

First we reformulate the semigroup $*$-algebra $\C[\Z^2]$.
Clearly, $\C[\Z^2]$ is generated by the functions $z, \ov{z}, z^{-1},\ov{z}^{\, -1}$ on the complex plane, that is, $\C[\Z^2]$ is the $*$-algebra $\C[z,\ov{z},z^{-1},\ov{z}^{-1}]$ of  complex Laurent polynomials in $z$ and $\ov{z}$. A vector space basis of this algebra is  $\{z^k\ov{z}^l; k,l\in \Z\}$. Writing $z=x_1+{\ii} x_2$ with $x_1,x_2\in \R$ we  have
$$
z^{-1}=\frac{x_1-{\ii} x_2}{x_1^2+x_2^2}~~~{\rm and}~~~{\ov{z}}^{\,-1}=\frac{x_1+{\ii} x_2}{x_1^2+x_2^2}.
$$
Hence $\C[\Z^2]$ is the complex unital $*$-algebra generated by the  hermitian functions
\begin{align}\label{bisgaard}
x_1, ~~x_2,~~ y_1:= \frac{x_1}{x_1^2+x_2^2}\,,~~y_2:=\frac{x_2}{x_1^2+x_2^2}
\end{align}
on $\R^2\backslash \{0\}$.
All four functions are unbounded on $\R^2\backslash \{0\}$ and we have
\begin{align}\label{yx2ide}
(y_1+{\ii}y_2)(x_1-{\ii }x_2)=1.
\end{align}
{\it Proof of Theorem \ref{bisgaardthm}:}\\
As above we identify $\C$ and $\R^2$ in the obvious way. As discussed at the end of Section \ref{propertiessmpandmpandfibretheorem}, the hermitian part of the complex $*$-algebra $\C [\Z^2]$ is a real algebra $\cA$.
First we determine the character set $\hat{\cA}$ of
$\cA$.
Obviously, the point evaluation at each point $x\in  \R^2\backslash \{0\}$ defines uniquely a character $\chi_x$ of $\cA$.  From (\ref{yx2ide}) it follows at once that there is no character $\chi$ on $\cA$ for which $\chi(x_1)=\chi(x_2)=0$. Thus,
$$
\hat{\cA}=\{\chi_x; x\in \R^2, x\neq 0\, \}.
$$
The three functions

$$h_1(x)=x_1y_1=\frac{x_1^2}{x_1^2+x_2^2}\,,~  h_2(x)=x_2y_2=\frac{x_2^2}{x_1^2+x_2^2}\,,~ h_3(x)=2x_1y_2=\frac{x_1x_2}{x_1^2+x_2^2}$$
are elements of $\cA$ and they are bounded on $\hat{\cA}\cong\R^2\backslash \{0\}$.
The same reasoning as  in the proof of Theorem \ref{mpshperesinrd} shows that the fibre algebra $\cA/\cI_\lambda$ for each nonempty fibre is a polynomial algebra in a single variable. Therefore, by Hamburger's theorem,  the preordering $\sum (\cA/\cI_\lambda)^2$ satisfies (MP) in $\cA/\cI_\lambda$ and so does $\sum \cA^2$ in $\cA$ by  Theorem \ref{fibreth1},(iii)$\to$(i). Since\, $\hat{\cA}\cong\R^2\backslash \{0\}=\C^\times$,  this gives the assertion.
\hfill $\Box$
\smallskip

{\bf Remark.} The  algebra $\cA$ generated by the four functions $x_1,x_2,y_1,y_2$ on $\C^\times $ is an interesting structure: The generators satisfy   the  relations
$$
x_1y_1+x_2y_2=1\quad {\rm and}\quad (x_1^2+x_2^2)(y_1^2+y_2^2)=1
$$
and there is a $*$-automorphism $\Phi$ of the real algebra $\cA$ (and hence of the complex $*$-algebra $\C[\Z^2]$) given by $\Phi(x_j)=y_j$ and $\Phi(y_j)=x_j$, $j=1,2$.

\bigskip

\noindent{\bf Acknowledgements.} The author would like to thank  T. Netzer  for valuable discussions on the subject of this paper and M. Wojtylak for useful comments.

\bibliographystyle{amsalpha}

\end{document}